\documentclass[11pt]{amsart}
\usepackage{latexsym,amssymb,amsmath,youngtab}
\usepackage{youngtab}
\usepackage{epsfig}
\usepackage{amsmath,amsthm,amssymb,amscd,MnSymbol}
\usepackage[mathscr]{eucal}
\usepackage{verbatim, hyperref}
\usepackage{color}
\newcommand*{\textlabel}[2]{%
  \edef\@currentlabel{#1}
  \phantomsection
  #1\label{#2}
}

\newtheoremstyle{custom}
  {3pt}
  {3pt}
  {\slshape}
  {}
  {\bfseries}
  {.}
  { }
   {}
\theoremstyle{custom}
\newtheorem{theorem}{Theorem}[section]
\newtheorem{proposition}[theorem]{Proposition}

\newtheorem{proposition/definition}[theorem]{Proposition/Definition}

\theoremstyle{definition}

\theoremstyle{remark}
\newtheorem{remark}[theorem]{Remark}

\newtheoremstyle{exercise}
  {3pt}
  {6pt}
  {}
  {}
  {\bfseries}
  {:}
  { }
   {}
\theoremstyle{exercise}
\newtheorem{exercise}[theorem]{Exercise}
\newtheoremstyle{exercises}
  {3pt}
  {6pt}
  {}
  {}
  {\bfseries}
  {:}
  {\newline}
   {}

\theoremstyle{exercise}
\newtheorem{exercises}[theorem]{Exercises}

\input epsf
\def\boxit#1{\vbox{\hrule height1pt\hbox{\vrule width1pt\kern3pt
  \vbox{\kern3pt#1\kern3pt}\kern3pt\vrule width1pt}\hrule height1pt}}

\def\trank{\text{rank}}

\def\bv{\bold v}

\def\BC{\mathbb C}

\def\BP{\mathbb P}
\def\pp#1{\mathbb P^{#1}}

\def\pp#1{{\mathbb P}^{#1}}
\def\tdim{{\rm dim}}

\def\hd{,...,}

\def\11{\mathbf 1}

\def\l{\lambda}
\def\a{\alpha}

\def\o{\omega}

\def\b{\beta}
\def\g{\gamma}
\def\s{\sigma}

\def\ot{{\mathord{ \otimes } }}
\def\op{{\mathord{\,\oplus }\,}}

\def\ra{{\mathord{\;\rightarrow\;}}}

\def\op{\oplus}
\def\BZ{\Bbb Z}

\def\op{\oplus}

\def\s{\sigma}
\def\t{\tau}

\def\a{\alpha}
\def\b{\beta}

\def\g{\gamma}
\def\l{\lambda}

\def\FS{\mathfrak  S}

\def\BP{\mathbb  P}
\def\BC{\mathbb  C}

\def\pp#1{\mathbb  P^{#1}}

\def\hd{, \hdots ,}

\def\pp#1{\mathbb  P^{#1}}

\def\ur{\underline {\bold R}}

\def\ra{\rightarrow}

\def\tend{\operatorname{End}}

\def\tdim{\operatorname{dim}}

\def\tlim{\lim}

\def\trank{\operatorname{rank}}

\def\bbb{{\bold{b}}}

\def\be{\begin{equation}}
\def\ene{\end{equation}}
\def\aaa{{\bold {a}}}
\def\bbb{{\bold {b}}}
\def\ccc{{\bold {c}}}

\newcommand{\Id}{\operatorname{Id}}
\newcommand{\Tr}{\operatorname{Tr}}

\newcommand{\Z}{\mathbb{Z}}

\def\trank{{\mathrm {rank}}}

\def\aaa{{\bold a}}\def\bbb{{\bold b}}\def\ccc{{\bold c}}

\def\uuu{\bold u}

\newcommand{\GL}{\operatorname{GL}}

\def\BL{\Bbb L}

\def\bv{\bold v}\def\bw{\bold w}

\begin{document}

\title[On the geometry of border rank algorithms]{On the geometry of border rank algorithms for  
  $n\times 2$ by $2\times 2$ matrix multiplication}
\author{J.M. Landsberg and Nicholas Ryder}
\begin{abstract}We make an in-depth study of the known border rank (i.e. approximate)
algorithms for the   matrix multiplication
tensor $M_{\langle n,2,2\rangle}\in \BC^{2n}\ot \BC^4\ot \BC^{2n}$ encoding the 
multiplication of an  $n\times 2$ matrix by a  $2\times 2$ matrix. 
 \end{abstract}
\thanks{Landsberg partially  supported by   NSF grant  DMS-1405348.}
\email{jml@math.tamu.edu, nick.ryder@berkeley.edu}
\maketitle

\section{Introduction}
This is the first of a planned series of articles examining the geometry of algorithms for matrix multiplication tensors.
Geometry has been used effectively in proving lower bounds for the complexity of matrix multiplication (see, e.g. \cite{Strassen505,v011a011}), and one goal
of this series is to initiate the use of geometry in proving upper bounds via practical algorithms for small matrix
multiplication tensors. 

A guiding principle is that if a tensor has symmetry, then there should be optimal expressions for it that  reflect that symmetry.
The matrix multiplication tensors have extraordinary symmetry. In this paper we examine algorithms, more precisely
{\it border rank algorithms} (see below for the definition),  that were originally found via numerical methods and  computer searches.  

Here is a picture illustrating  the geometry of an algorithm due to Alekseev-Smirnov that we discuss in \S\ref{tbcrls3}:

\begin{figure}[!htb]\begin{center}\label{BCLRpic}
\includegraphics[width=9.6cm]{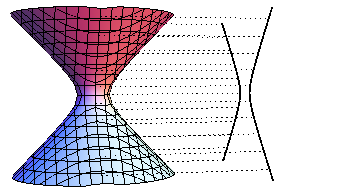} 
\caption{\small{A quadric surface, a plane conic curve and a one parameter family of lines connecting them.}}  
\end{center}
\end{figure}

\medskip

A tensor $T\in \BC^{\aaa}\ot \BC^{\bbb}\ot \BC^{\ccc}$ has {\it rank one} if there exist
$a\in \BC^{\aaa}$, $ b\in   \BC^{\bbb}$ and $c\in  \BC^{\ccc}$ such that
$T=a\ot b\ot c$. A {\it rank $r$ expression} for a tensor  $T\in \BC^{\aaa}\ot \BC^{\bbb}\ot \BC^{\ccc}$ is a collection of
 rank one tensors $T_1\hd T_r$ such that $T=T_1+\cdots + T_r$. 
A {\it border rank $r$ algorithm} for $T$ is an expression
$$
T=\tlim_{t\ra 0} T_1(t)+\cdots + T_r(t)
$$
with each $T_j(t)$ of rank one, and for $t>0$ the $T_j(t)$ are linearly independent. 
The first interesting border rank algorithm was found by Bini-Capovani-Lotti-Romani (BCLR)
\cite{MR534068}, essentially by accident:  After Strassen's remarkable discovery \cite{Strassen493}
of a rank seven expression for the $2\times 2$ matrix multiplication tensor, and
Winograd's proof shortly afterward \cite{win553} that no rank six expression existed, BCLR attempted to determine
if the rank of the $2\times 2$ matrix multiplication tensor where  
an entry of one of the matrices is zero could have an expression of rank less than six.
They used an alternating least squares iteration scheme on a computer.
Instead of finding such an expression, they found the border rank expression \eqref{bcrlexpr} below.
(That some tensors have border rank lower than rank was known to Terracini in 1911 \cite{Terracini1}, if not earlier, but not
to the computer science community.)
Later Smirnov \cite{MR3343116,MR3146566} and Alekseev-Smirnov \cite{AlSmir}, using similar,  but
more sophisticated methods, found further border rank
algorithms for small matrix multiplication tensors. 

In this paper we describe  
geometry in these algorithms, with   very satisfactory answers in first few cases  and
successively weaker results as the tensors get larger.
We begin, in \S\ref{mmultbcrl} with a review of the matrix multiplication and BCLR-type tensors. 
We discuss the known upper and lower bounds on their border ranks in \S\ref{brbnds}.
In \S\ref{brgeom} and \S\ref{seggeom} we respectively discuss the geometry of border rank
algorithms and the Segre variety. In sections \S\ref{tbcrl}--\ref{m333} we analyze the various algorithms.
We conclude with  a brief discussion of the uniquess of
the BCLR algorithms  in \S\ref{uniquerems}.

\subsection*{Notation} We let $A,B,C,U,V,W$ denote  complex vector spaces of dimensions $\aaa,\bbb,\ccc,\uuu,\bv,\bw$.
 If $W$ is a vector space then
$\BP W$ is the associated projective space of lines through the origin:
$\BP W=(W\backslash 0)/\sim$ where $w_1\sim w_2$ if $w_1=\l w_2$ for some 
nonzero complex number $\l $. Write $[w]\in \BP W$ for
the equivalence class of $w\in W\backslash  0$ and if $X\subset \BP W$, let $\hat X\subset W$ denote
the corresponding cone in $W$.  The linear span of vectors $w_1\hd w_s$ is denoted
$\langle w_1\hd w_s\rangle$, and the span of $[w_j]\in \BP W$ is similarly denoted
$\langle [w_1]\hd [w_s]\rangle\subset \BP W$.

The set of rank one tensors in $\BP (A\ot B\ot C)$ is
isomorphic to $\BP A\times \BP B\times \BP C$, and the
inclusion into $\BP (A\ot B\ot C)$
is denoted $Seg(\BP A\times \BP B\times \BP C)\subset\BP (A\ot B\ot C)$
and called the {\it Segre variety}.

\subsection*{Acknowledgments}
We thank F. Gesmundo for calculating the Lie algebra of the 
stabilizer of $T_{BCLR}$ and the limiting $5$ and $10$-planes in the BCLR algorithms.
This paper is the result of a project associated to   a course at UC Berkeley fall 2015 given  by 
the first author and attended by the second  as part
of  a semester long program
 {\it Algorithms and Complexity in Algebraic Geometry} at the 
 Simons Institute for the Theory of Computing. The authors thank the Institute
 for making this paper possible.

\section{Matrix multiplication and the BCLRS tensors}\label{mmultbcrl}
The matrix multiplication tensor is   
\be \label{mmultis}M_{\langle \uuu,\bv,\bw\rangle}=M_{\langle U,V,W\rangle}=\Id_U\ot \Id_V\ot \Id_W
\in (U^*\ot V)\ot (V^*\ot W)\ot (W^*\ot U)
\ene
see \cite[\S 2.5.2]{MR2865915}. 
We write 
$$M_{\langle \uuu,\bv,\bw\rangle}=\sum_{k=1}^{\bw}\sum_{j=1}^\bv\sum_{i=1}^\uuu x^i_j\ot y^j_k\ot z^k_i=\sum_{k=1}^{\bw}\sum_{j=1}^\bv\sum_{i=1}^\uuu
(u^i\ot v_j)\ot (v^j\ot w_k)\ot (w^k\ot u_i)
$$
where $\{ u_i\}$ is a basis of $U$ with dual basis $\{ u^i\}$ and similarly for $V,W$.

Define the generalized Bini-Capovani-Lotti-Romani-Smirnov tensor, corresponding
to $m\times 2$ by $2\times 2$ matrix multiplication with the $x^1_1$ entry set equal
to zero:
$$
T_{BCLRS,m}:=
M_{\langle m,2,2\rangle} - x^1_1\ot (y^1_1\ot z^1_1+y^1_2\ot z^2_1)\in \BC^{2m-1}\ot \BC^4\ot \BC^{2m}.
$$
(In the original tensor BCLR set the $x^2_1$ entry to zero. We set $x^1_1$
 equal to zero to facilitate comparisons between different values of $m$.)

\section{Border rank    bounds for the BCLRS   tensors}\label{brbnds}

The following observation  dates back to \cite{MR534068}:
\begin{proposition}\label{twoaddprop} If $\ur(T_{BCLRS,m})=r$ and
$\ur(T_{BCLRS,m'})=r'$, then setting $n=m+m'-1$,
$\ur(M_{\langle n,2,2\rangle})\leq r+r'$.
\end{proposition}
 
Upper bounds on the border ranks
of these tensors are:    
   $\ur(M_{\langle n,2,2\rangle})\leq 3n+\lceil \frac n7\rceil$  for all   $n$, 
  \cite{MR534068,AlSmir,MR3343116}.
  (Equality holds for $n=1$ (classical), and
$n=2$ \cite{MR2188132},  see \cite{HILsec} for a better proof.) 
  
  It would be reasonable to expect that 
 the BCLR, Alekseev-Smirnov,  and Smirnov algorithms     generalize to all $m$, so that
  $\ur(T_{BCLRS,m})\leq 3m-1$.
If that happens,    Proposition \ref{twoaddprop} would imply
that $\ur(  M_{\langle n,2,2\rangle})\leq 3n+1$ for all $n$.
\begin{proposition} $\ur(T_{BCLR})= 5$ and for $m>2$, $  \ur(T_{BCLRS,m})\geq 3m-2$
\end{proposition}
\begin{proof}
The upper bound for $T_{BCLR}$ comes from \cite{MR534068}.

For the lower bounds we use Strassen's equations \cite{Strassen505}: Let $T\in A\ot B\ot C$ be  such that there exists $\a\in A^*$ with $T(\a)\in B\ot C$ of
maximal rank. Assuming $\bbb=\tdim B\leq \tdim C$, take $C'\subset C^*$ with $\tdim C'=\bbb$, and such that
$\trank T(\a)|_{B^*\times  C'}=\bbb$. Then use $T(\a)|_{B^*\ot C'}$ to identify $B^*\ot C'\simeq \tend(B)$,
Strassen's equations state that for all $X_1,X_2\in T(A^*)|_{\tend(B)}$,
letting $[X_1,X_2]$ denote their commutator, 
$$
\ur(T)\geq \frac 12\trank[X_1,X_2]+\bbb.
$$

Consider 
$$T_{BCLRS,m}(B^*)=
\begin{pmatrix}
y^2_1&y^2_2    & &  & &   &   \\
& &  y^1_1&y^1_2    & &   &   \\
& &  y^2_1&y^2_2    & &   &   \\
& &   & &\ddots   & &   &  \\
& &  &   & &      y^1_1&y^1_2 \\
& &  &   & &       y^2_1&y^2_2
\end{pmatrix}
$$
This is a $2m\times (2m-1)$ matrix of linear forms. Take the submatrix setting the first
column to zero to have a square matrix. Making generic choices, the first 
$1\times 1$ block will not contribute to the commutator but all other blocks contribute
a rank two matrix.
We conclude 
$\ur(T_{BCLRS,m})\geq \frac 12(2m-2)+2m-1=3m-2$.
When $m=2$ we can do a little better by considering $T_{BCLR}(A^*)\subset \BC^4\ot \BC^4$. Under generic choices the commutator of each
$2\times 2$  block contributes a rank one matrix
and we conclude $\ur(T_{BCLR})\geq \frac 12(2)+4=5$.
\end{proof}

Similarly, Strassen's equations imply
$ 
\ur(M_{\langle n,2,2\rangle})\geq 3n .
$ 
In summary:
$$
3n\leq \ur(M_{\langle n,2,2\rangle})\leq 3n+\lceil\frac n7\rceil
$$
and there is evidence for an upper bound of $3n+1$.

 \section{What is a border rank algorithm?}\label{brgeom}

Usually a border rank algorithm is presented as
$$
T=\tlim_{t\ra 0} T_1(t)+\cdots + T_r(t)
$$
with each $T_j(t)$ of rank one and the $T_j(t)$ linearly independent when $t\neq 0$. To work   geometrically
we   focus on the curve of $r$-planes $\langle T_1(t)\hd T_r(t)\rangle
\subset G(r,A\ot B\ot C)$ that the border rank algorithm defines.
Here, for a vector space $V$, $G(r,V)$ denotes the Grassmannian of
$r$-planes through the origin in $V$.

For the purposes of this paper, a border rank algorithm is a point  $E\in G(r,A\ot B\ot C)$ such that
$T\in E$ and there exists a curve $E_t$ limiting to $E$ with
$E_t$ spanned by $r$ rank one elements for all $t>0$.

\begin{remark} More precisely a border rank algorithm should be thought
of as an $h$-jet of a curve in the Grassmannian that is the $h$-jet of some
curve spanned by rank one elements. 
\end{remark}

\begin{remark} This discussion generalizes to arbitrary secant varieties,
see \cite{LMborderalg}.
\end{remark} 

A border rank algorithm will not be a rank algorithm
when $T_1(0)\hd T_r(0)$ fail to be linearly independent. Say this is the case  and no subset of the points
fails to be linearly independent. Then $T$   can be any point in 
$\langle \hat T_{T_1(0)}Seg(\BP A\times \BP B\times \BP C)\hd \hat T_{T_r(0)}Seg(\BP A\times \BP B\times \BP C)\rangle$,
where $\hat T_xSeg(\BP A\times \BP B\times \BP C)\subset A\ot B\ot C$ denotes the affine tangent space to the Segre at $x$, 
see \cite[\S 10.8]{MR2865915}. In this case we call the algorithm {\it first order}.

A {\it second order algorithm} occurs when the sum of the tangent vectors fails to be linearly independent
from the original $r$ vectors (which themselves fail to be linearly independent). In this case, for
each of the tangent vectors appearing, there is its image under the second fundamental form as described
in Equation \eqref{iivect} below, and $T$ is the sum of these vectors
plus   any point in the sum of the tangent spaces.

Higher order algorithms exist, but we do not discuss their geometry in this paper.

\section{On the geometry of the Segre variety}\label{seggeom}

In order to have a border rank algorithm one must have
$r$ points on the Segre that fail to be linearly independent. 
The most na\"\i ve way to attain this is to have a point appearing at least twice.
For example, the most classical tensor with border rank lower than rank is
$$
a_1\ot b_1\ot c_2+a_1\ot b_2\ot c_1+a_2\ot b_1\ot c_1=
\tlim_{t\ra 0}\frac 1t[
(a_1+ta_2)\ot (b_1+tb_2)\ot (c_1+tc_2)- a_1\ot b_1\ot c_1]
$$
where both points limit to $a_1\ot b_1\ot c_1$.

The next most na\"\i ve limits are when $r $ points all lie on an $r-1$-plane.
This is the case for Sch\"onhage's algorithm for the sum of two disjoint tensors \cite{MR623057}.

The configurations that   arise  in border rank algorithms
for $T_{BCLRS,m}$  are more interesting. What follows are geometric preliminaries needed
to describe them.

\medskip

We first describe lines on Segre varieties.
There are three types: $\a$-lines, which are of the form
$\BP( \langle a_1,a_2\rangle\ot b\ot c)$ for some
$a_j\in A$, $b\in B$, $c\in C$, and the other two types are defined similarly and
called 
$\b$ and $\g$ lines.

\medskip

Given two    lines $L_{\b},L_{\g}\subset Seg(\BP A\times \BP B\times \BP C)$ respectively of type $\b,\g$, if they
do not intersect, then $  \langle L_\b, L_\g\rangle= \pp 3$ and if the lines are general, furthermore
$\langle L_\b,L_\g\rangle\cap Seg(\BP A\times \BP B\times \BP C)=L_\b\sqcup L_\g$.

However if $L_{\b}=\BP(a\ot\langle b_1,b_2\rangle\ot c)$ and $L_{\g}=\BP (a'\ot b\ot \langle c_1,c_2\rangle)$
with $b\in \langle b_1,b_2\rangle$ and $c\in\langle c_1,c_2\rangle$, then they still span a $\pp 3$ but 
$\langle L_\b,L_\g\rangle\cap Seg(\BP A\times \BP B\times \BP C)=L_\b\sqcup L_\g\sqcup L_{\a}$,
where $L_{\a}=\BP (\langle a,a'\rangle\ot  b\ot c)$, and $L_{\a}$ intersects both $L_{\b}$ and $L_{\g}$. 

\medskip

  Let $x,y,z\in Seg(\BP A\times \BP B\times \BP C)$
be distinct points that 
all lie on a line $L\subset Seg(\BP A\times \BP B\times \BP C)$. Then
$$\hat T_xSeg(\BP A\times \BP B\times \BP C)
\subset \langle \hat T_ySeg(\BP A\times \BP B\times \BP C),\hat T_zSeg(\BP A\times \BP B\times \BP C)\rangle .
$$
In fact, the analogous statement is true for lines on any cominuscule variety, see \cite[Lemma 3.3]{MR3239293}.
Because of this, it will be more geometrical to refer
to 
$\hat T_LSeg(\BP A\times \BP B\times \BP C):=\langle \hat T_ySeg(\BP A\times \BP B\times \BP C),\hat T_zSeg(\BP A\times \BP B\times \BP C)\rangle$, as the choice
of $y,z\in L$ is irrelevant, at least for first order algorithms.

\medskip

The matrix multiplication tensor 
$M_{\langle U,V,W\rangle}$   \eqref{mmultis}
endows  $A,B,C$ with  additional structure, e.g.,
$B=V^*\ot W$, so there are two types of distinguished  $\b$-lines (corresponding
to lines of rank one matrices), call them $(\b, \nu^*)$-lines and $(\b,\o)$-lines, where,
e.g., a $\nu^*$-line is of the form $\BP (a\ot (\langle v^1,v^2\rangle\ot w)\ot c)$, and
among such lines there are further distinguished  ones where moreover both
$a$ and $c$ also have rank one. Call such further distinguished lines {\it special} $(\b,\nu^*)$-lines.

\section{$T_{BCLR}$}\label{tbcrl}

Here  $A\subset U^*\ot V$ has dimension three, so we don't have
the full space of $2\times 2$ matrices.

What follows is a slight modification of the BCLR algorithm.
We label the points such that $x^1_1$ is set equal to zero. The main difference
is that in the original all five points moved, but here one is stationary.

\begin{align*}
 p_1(t) =    x^1_2  \otimes (y^2_2 + y^2_1) \otimes (z^2_2 + t z^1_1)\\
 p_2(t) =   -(x^1_2 - t x^2_2) \otimes  y^2_2  \otimes (z^2_2 + t(z^1_1+ z^2_1))\\
 p_3(t) =    x^2_1  \otimes (y^2_1 + t y^1_2) \otimes (z^2_2 + z^1_2)\\
 p_4(t) =   (x^2_1 - t x^2_2) \otimes ( - y^2_1 + t y^1_1- t y^1_2) \otimes  z^1_2 \\
 p_5(t) =   -(x^2_1 + x^1_2) \otimes  y^2_1 \otimes  z^2_2 
 \end{align*}
and
\be\label{bcrlexpr}
T_{BCLR}= \frac 1t[p_1(t)+\cdots +p_5(t)].
\ene
Let $E^{BCLR}=\tlim_{t\ra 0}\langle p_1(t)\hd p_5(t)\rangle \in G(5,A\ot B\ot C)$. 

\begin{theorem} Notations as above.
In the BCLR algorithm   $E^{BCLR} \cap Seg(\BP A\times \BP B\times \BP C)$ is the union of three
lines: $L_{12,(\b,\o)}$, which is a special $(\b,\o)$-line, $L_{21,(\g,\o^*)}$, which is a special
$(\g,\o^*)$-line, and $L_{\a}$, which is an $\a$-line with rank one  $b$ and $c$ points.
Moreover,  the 
$C$-point of $L_{12,(\b,\o)}$ lies in the $\o^*$-line of $L_{21,(\g,\o^*)}$,
the $B$-point of $L_{21,(\g,\o^*)}$ lies in the $\o$-line of $L_{12,(\b,\o)}$
and   $L_{\a }$ is the unique line on the Segre intersecting 
$L_{12,(\b,\o)}$ and $L_{21,(\g,\o^*)}$ (and thus it is contained in their span). 

Explicitly:
\begin{align*}
 &L_{12,(\b,\o)} = x^1_2 \otimes (v^2 \otimes W) \otimes z^1_2\\
&L_{21,(\g,\o^*)} = x^2_1 \otimes y^2_2 \otimes (W^* \otimes u_2)\\
&L_{\a } = \langle x^2_1, x^1_2\rangle \otimes y^2_2 \otimes z^1_2.
\end{align*}

Furthermore,  $E^{BCLR}=\langle T_{BCLR}, L_{12,(\b,\o)},L_{21,(\g,\o^*)}\rangle$ and
$$T_{BCLR}\in \langle \hat T_{L_{12,(\b,\o)}}Seg(\BP A\times \BP B\times \BP C),
\hat T_{L_{12,(\b,\o)}}Seg(\BP A\times \BP B\times \BP C)\rangle.
$$
\end{theorem}

\begin{proof}
Write $p_j=p_j(0)$. Then (up to sign, which is irrelevant for geometric considerations)
\begin{align*}
p_1= & x^1_2              \ot (y^2_2 + y^2_1)          \ot z^2_2         \\
p_2= & x^1_2                 \ot y^2_2                    \ot z^2_2 \\
p_3= & x^2_1                \ot y^2_1                \ot (z^2_2 + z^1_2)      \\
p_4= & x^2_1                 \ot  y^2_1                \ot z^1_2        \\
p_5= & (x^2_1 + x^1_2) \ot  y^2_1                       \ot z^2_2         \\
\end{align*} 
The configuration of lines is as follows: 
\begin{align*}
p_1, p_2 \in  x^1_2  \otimes (v^2 \otimes W) \otimes z^2_2\\
 p_3, p_4 \in  x^2_1  \otimes  y^2_1 \otimes ( W^* \otimes u_2)\\
 p_5 \in  \langle x^1_2, x^2_1 \rangle  \otimes y^2_1 \otimes z^2_2.\end{align*}
 
 To see there are no other points in $E^{BCLR}\cap Seg(\BP A\times \BP B\times \BP C)$,
 first note that any such point would have to lie on
 $Seg(\BP \langle x^1_2,x^2_1\rangle\times \BP \langle y^2_1,y^2_2\rangle
 \times \BP \langle z^1_2,z^2_2\rangle)$ because there is no  way to
 eliminate the  rank two $x^2_2\ot (y^2_1\ot z^1_2+y^2_2\ot z^2_2)$ term in $T_{BCLR}$ with a linear 
 combination of $p_1\hd p_4$.
 Let $[(sx^1_2+tx^2_1)\ot (uy^2_2+vy^2_1)\ot (pz^2_2+qz^1_2)]$
 be an arbitrary point on this variety. To have it be in the 
 span of $p_1\hd p_4$ it must satisfy the equations
 $suq=0$, $svq=0$, $tuq=0$, $tup=0$. Keeping in mind that one cannot
 have $(s,t)=(0,0)$, $(u,v)=(0,0)$, or $(p,q)=(0,0)$,  we conclude the
 only solutions are the three lines already exhibited. 
 
We have
\begin{align*}
 p_1(0)' &=    x^1_2  \otimes (y^2_2 + y^2_1) \otimes   z^1_1 \\
 p_2(0)' &=       x^2_2  \otimes  y^2_2  \otimes  z^2_2  
 - x^1_2  \otimes  y^2_2  \otimes   (-   z^2_1 +   z^1_1)\\
 p_3(0)'& =    x^2_1  \otimes  y^1_2  \otimes (z^2_2 + z^1_2)\\
 p_4(0)' &=   x^2_2  \otimes   y^2_1  \otimes  z^1_2
 + x^2_1   \otimes (  y^1_1-   y^1_2) \otimes  z^1_2\\
 p_5(0)'& =   0
 \end{align*}
 
Then $T_{BCLR}=(p_1'+p_2')+(p_3'+p_4')$ where $p_1'+p_2'\in T_{L_{12,(\b,\o)}}Seg(\BP A\times \BP B\times \BP C)$
and $p_3'+p_4'\in T_{L_{21,(\g,\o^*)}}Seg(\BP A\times \BP B\times \BP C)$.
\end{proof}

 \begin{remark}
 When we allow $GL_4^{\times 3}\rtimes \FS_3$ to act on $\BC^4\ot \BC^4\ot \BC^4$,
 the subgroup preserving the
  tensor  $M_{\langle 2, 2, 2\rangle}$   
  is $SL_2\times SL_2\times SL_2\rtimes \BZ_3$, where
  the $SL_2$'s are respectively $SL(U),SL(V),SL(W)$
  and the $\BZ_3$ is most easily see by   viewing matrix multiplication as trilinear map which   sends three matrices to the trace of their product: 
  $M, N, L \mapsto \Tr(MNL)$. 
  The $\BZ_3$-symmetry follows as 
  $\Tr(MNL)=\Tr(NLM)$. The full  symmetry group includes
  a $\BZ_2\rtimes \BZ_3$ action as $\Tr(X)=\Tr(X^T)$, where $X^T$ denotes the
  transpose of $X$.   By removing $x^1_1$ from our tensor, we lose the
  $\BZ_3$, but retain a $\Z_2$ action which corresponds to $\Tr(MNL) = \Tr(M^TL^TN^T)$. 
  Similarly we lose our $\GL(U) \times \GL(V)$ symmetry but retain our $\GL(W)$ action. 
  By composing our discrete $\Z_2$ symmetry with another $\Z_2$ action which switches the basis vectors of $W$,
  the action swaps $p_1(t) + p_2(t)$   with $p_3(t) + p_4(t)$ and $L_{12, (\b, \o)}$   with $L_{21,(\g,\o^*)}$. This $\Z_2$ action fixes $p_5(t)$. 
\end{remark}
 
 \begin{remark}
Note that it is important that $p_5$ lies neither on $L_{12,(\b,\o)}$ nor on $L_{21,(\g,\o^*)}$, so that no subset of the five points
lies in a linearly degenerate position to enable us to have tangent vectors coming from all five points, but we emphasize that
any such point (i.e., any point on the line $L_{\a}$ not on the original lines) would have worked equally well, so the
geometric object is this configuration of lines. 
 \end{remark}

\section{$T_{BCLRS,3}$}\label{tbcrls3}

Here  is the algorithm in \cite[Thm. 2]{AlSmir} only changing the element set to zero to $x^1_1$ (it is $x^3_2=a^3_2$ in \cite{AlSmir}).

\begin{align*}
&p_1(t) =   (\frac{-1}{2}t^2 x^3_2 - \frac{1}{2} t x^2_1 + x^2_1) \otimes (-y^2_1 + y^2_2 + t y^1_1) \otimes (z^1_3 + t z^1_2)\\
&p_2(t) =  (x^2_1 + \frac{1}{2} x^1_2) \otimes (y^2_1 - y^2_2) \otimes (z^1_3 + z^2_3 + t z^1_2 + t z^2_2)\\
&p_3(t) =   (t^2 x^3_2 + t x^3_1 - \frac{1}{2}t  x^2_2 - x^2_1) \otimes (y^2_1 + y^2_2 + t y^1_2) \otimes  z^2_3 \\
&p_4(t) =   (\frac{1}{2} t^2 x^3_2 - t x^3_1 - \frac{1}{2} t x^2_2 + x^2_1) \otimes (y^2_1 + y^2_2 - t y^1_1) \otimes  z^1_3 \\
&p_5(t) =   (-t^2 x^3_2 + t x^2_2 - x^1_2) \otimes y^2_1 \otimes (z^2_3 + \frac{1}{2} t z^1_2 + \frac{1}{2} t z^2_2 - t^2 z^1_1)\\
&p_6(t) =   (\frac{1}{2}  t x^2_2 + x^2_1) \otimes (-y^2_1 + y^2_2 + t y^1_2) \otimes (z^2_3 + t z^2_2)\\
&p_7(t) =   ( -tx^3_1 + x^2_1 + \frac{1}{2} x^1_2) \otimes (y^2_1 + y^2_2) \otimes (-z^1_3 + z^2_3)\\
&p_8(t) =   (t x^2_2 + x^1_2) \otimes y^2_2 \otimes (z^1_3 + \frac{1}{2} t z^1_2 + \frac{1}{2} t z^2_2 + t^2 z^2_1)
\end{align*}

Then 
$$T_{BCLRS,3}=\frac 1{t^2}[p_1(t)+\cdots + p_8(t)].$$
Let $E^{AS,3}=\tlim_{t\ra 0}\langle
p_1(t)\hd  p_8(t)\rangle\in G(8,A\ot B\ot C)$.

\begin{theorem} Notations as above.
In the Alekseev-Smirnov  algorithm for $T_{BCLRS,3}$,  $E^{AS,3}\cap Seg(\BP A\times \BP B\times \BP C)$ is the union of 
two irreducible algebraic surfaces, both abstractly isomorphic to  $\pp 1\times \pp 1$:  
The first is a sub-Segre variety: 
$$Seg_{21,(\b,\o),(\g,\o^*)}:=
[x^2_1]\times \BP (v^2\ot W)\times \BP (W^*\ot u_3),
$$
The second,  $\BL_{\a}$   is a union of lines passing through $Seg_{21,(\b,\o),(\g,\o^*)}$
and the plane conic curve:
$$
C_{12,(\b,\o),(\g,\o^*)}:=\BP(\cup_{[s,t]\in \pp 1}  x^1_2\ot (sy^2_1-ty^2_2)\ot(sz^2_3+tz^1_3)).
$$
The three varieties $C_{12,(\b,\o),(\g,\o^*)}$, $Seg_{21,(\b,\o),(\g,\o^*)}$, and $\BL_{\a}$ respectively play
roles analogous to the 
lines  $L_{12,(\b,\o)}$, $L_{21,(\g,\o^*)}$,   and $L_{\a}$,  as described below. 
\end{theorem}

\begin{figure}[!htb]\begin{center}\label{BCLRptpic}
\includegraphics[width=9.6cm]{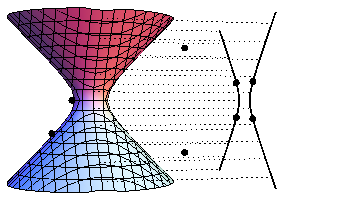} 
\caption{\small{The curve  $C_{12,(\b,\o),(\g,\o^*)}$ with its four points,
the surface $Seg_{21,(\b,\o),(\g,\o^*)}$,   with its 
four points (only two of which are visible), and 
the surface $\BL_{\a}$ with its  two points
   which don't lie on either the curve or surface $Seg_{21,(\b,\o),(\g,\o^*)}$.  }}  
\end{center}
\end{figure}

\begin{proof}
 The limit points are (up to sign):
\begin{align*}
p_1=& x^2_1              \ot  (y^2_1 - y^2_2)    \ot z^1_3         \\
p_3= & x^2_1    \ot (y^2_1 + y^2_2 )             \ot z^2_3         \\
p_4= & x^2_1    \ot (y^2_1 + y^2_2  )       \ot z^1_3         \\
p_6= & x^2_1   \ot ( y^2_1 - y^2_2  )        \ot z^2_3         \\
& \\
p_5= &  x^1_2   \ot  y^2_1                \ot z^2_3         \\
p_8= & x^1_2    \ot y^2_2           \ot z^1_3         \\
& \\
p_2= & (x^2_1 +  \frac{1}{2} x^1_2 )   \ot  (y^2_1 - y^2_2  )      \ot (z^1_3 + z^2_3) \\
p_7= & (x^2_1 + \frac{1}{2}x^1_2 ) \ot (y^2_1 + y^2_2 )      \ot (z^1_3 - z^2_3)\\
\end{align*}
Just as with $T_{BCLR}$, the limit points all lie on a $Seg(\pp 1\times \pp 1\times \pp 1)$, in fact  
the \lq\lq same\rq\rq\ $Seg(\pp 1\times \pp 1\times \pp 1)$. Pictorially the Segres are:
$$
\begin{pmatrix}
0 & * \\ *   \end{pmatrix}\times 
\begin{pmatrix}
  &   \\   * & *\end{pmatrix}
  \times
\begin{pmatrix}
   & * \\   & *\end{pmatrix}
  $$
 for $T_{BCLRS,2}$ and  
  $$
\begin{pmatrix}
0 & * \\ * & \\  & \end{pmatrix}\times 
\begin{pmatrix}
  &   \\   * & *\end{pmatrix}
  \times
\begin{pmatrix}
  &  & * \\   & & *\end{pmatrix}
  $$
for $T_{BCLRS,3}$.
Here $E^{AS,3}\cap Seg(\BP A\times \BP B\times \BP C)$ is the union of a one-parameter
family of lines $\BL_{\a}$  passing through a plane conic  and a special $\pp 1\times \pp 1$:   $Seg_{21,(\b,\o),(\g,\o^*)}:=
[x^2_1]\times \BP (v^2\ot W)\times \BP (W^*\ot u_3)$ (which contains $p_1,p_3,p_4,p_6$).
To define the family and make the similarity with the BCLR case clearer, first define 
  the plane conic curve
$$
C_{12,(\b,\o),(\g,\o^*)}:=\BP(\cup_{[s,t]\in \pp 1}  x^1_2\ot (sy^2_1-ty^2_2)\ot(sz^2_3+tz^1_3)).
$$
The points $p_5,p_8$ lie on this conic (respectively the values $(s,t)=(1,0)$ and $(s,t)=(0,1)$).
Then define the variety
$$
\BL_{\a}:= \BP(\cup_{[\s,\t]\in \pp 1} \cup_{[s,t]\in \pp 1}(\s   x^1_2+\t x^2_1) \ot (sy^2_1-ty^2_2)\ot(sz^2_3+tz^1_3)),
$$
which is a one-parameter family   of lines intersecting  the conic and  the special $\pp 1\times \pp 1$.
The points $p_2,p_7$ lie on $\BL_{\a}$ but not on the conic. Explicitly 
$p_2$ (resp.  $p_7$) is          the point  corresponding to the values $(\s,\t)=(1,\frac 12)$ and  $(s,t)=(1,1)$ (resp.   $(s,t)=(1,-1)$).

The analog of $L_{\a}$ in the $T_{BCLR}$ algorithm  is $\BL_{\a}$,
and $C_{12,(\b,\o),(\g,\o^*)}$ and $Seg_{21,(\b,\o),(\g,\o^*)}$ are the analogs of the lines  $L_{12,(\b,\o)},L_{21,(\g,\o^*)}$.
(A difference here is that $C_{12,(\b,\o),(\g,\o^*)}\subset \BL_{\a}$.)

The span  of the   configuration is the span of a $\pp 2$ (the span of the conic)  and
a $\pp 3$ (the span of the $\pp 1\times \pp 1$), i.e., a $\pp 6$. 

The proof that these are the only points in the intersection is similar to
the BCLR case. \end{proof}

\begin{remark}
We expect that just as with $T_{BCLR}$, the particular $8$ points in this configuration one uses
in the limit are irrelevant  as long as they are sufficiently general that no seven of them fail to be linearly independent. 
\end{remark}

\medskip

The tangent vectors to a point $[a\ot b\ot c]\in Seg(\BP A\times \BP B\times \BP C)$
are of the form $a'\ot b\ot c+ a\ot b'\ot c+ a\ot b\ot c'$. The following chart gives
the vectors $(a',b',c')$ for the tangent vectors that appear in the algorithm. Blank spaces
correspond to a zero vector:
$$\begin{array}{cccc}
p_1: & \frac{-1}{2} x^2_2           & y^1_1 & z^1_2 \\
p_3: & x^3_1 - \frac{1}{2} x^2_2    & y^1_2 &       \\
p_4: & -x^3_1 - \frac{1}{2} x^2_2   & -y^1_1&   \\
p_6: & \frac{1}{2} x^2_2            & y^1_2 & z^2_2 \\
&\\
p_5: & x^2_2                        &       & \frac{1}{2} z^1_2 + \frac{1}{2} z^2_2 \\
p_8: & x^2_2                        &       & \frac{1}{2} z^1_2 + \frac{1}{2} z^2_2 \\
& \\
p_2: &                              &       & z^1_2 + z^2_2\\
p_7: & -x^3_1                       &       &      . \\
\end{array}$$

There are two types of points that can appear at second order:  ordinary tangent vectors, and vectors
arising from the {\it second fundamental form}. The latter must appear: if
a tangent vector $a'\ot b\ot c+ a\ot b'\ot c+ a\ot b\ot c'$ appears at first order, then the vector
\be\label{iivect}a'\ot b'\ot c+ a\ot b'\ot c'+ a'\ot b\ot c'
\ene
must appear at second order, see \cite{MR3239293}. The following chart gives
the new ordinary tangent vectors appearing at second order in the same format as the tangent vectors above:
$$\begin{array}{cccc}
p_1: & \frac{-1}{2} x^3_2   &   &   \\
p_3: &  x^3_2    &   &   \\
p_4: & \frac{1}{2} x^3_2               &   &   \\
p_6: &  &   &\\
& \\
p_5: &    -x^3_2                  &   &   -z^1_1 \\
p_8: &  &   & z^2_1\\
& \\
p_2: &                &   &   \\
p_7: &  &   & .\\
\end{array}$$
Pictorially, the order entries are reached at (which coincides with the expression for $T_{BCLR}$ when one truncates) is 
 $$
\begin{pmatrix}
X & 0 \\ 0 & 1\\ 1 &2 \end{pmatrix}\times 
\begin{pmatrix}
 1 & 1  \\   0 & 0\end{pmatrix}
  \times
\begin{pmatrix}
 2 & 1 & 0\\   2&1 & 0\end{pmatrix}.
  $$

Explicitly:
$$\begin{array}{cc}
p_1(0)'= & \frac{-1}{2} x^2_2 \otimes (-y^2_1 + y^2_2) \otimes z^1_3 + x^2_1 \otimes y^1_1 \otimes z^1_3 + x^2_1 \otimes (-y^2_1 + y^2_2) \otimes z^1_2\\
p_3(0)'= & (x^3_1 - \frac{1}{2} x^2_2) \otimes (y^2_1 + y^2_2) \otimes z^2_3 + x^2_1 \otimes y^1_2 \otimes z^2_3 \\
p_4(0)'= & (-x^3_1 - \frac{1}{2}x^2_2) \otimes (y^2_1 + y^2_2) \otimes z^1_3 + x^2_1 \otimes -y^1_1 \otimes z^1_3 \\
p_6(0)'= & \frac{1}{2} x^2_2 \otimes (-y^2_1 + y^2_2) \otimes z^2_3 + x^2_1 \otimes y^1_2 \otimes z^2_3 + x^2_1 \otimes (-y^2_1 + y^2_2) \otimes z^2_2 \\
& \\ 
p_5(0)'= & x^2_2 \otimes y^2_1 \otimes z^2_3   -x^1_2 \otimes y^2_1 \otimes (\frac{1}{2} z^1_2 + \frac{1}{2} z^2_2) \\
p_7(0)'= & (-x^3_1 + \frac{1}{2} x^1_2) \otimes (y^2_1 + y^2_2) \otimes (-z^1_3 + z^2_3) \\
& \\
p_2(0)'= & (x^2_1 + \frac{1}{2} x^1_2) \otimes (y^2_1 - y^2_2) \otimes (z^1_2 + z^2_2)\\
p_8(0)'= & x^2_2 \otimes y^2_2 \otimes z^1_3 + x^1_2 \otimes y^2_2 \otimes ( \frac{1}{2} z^1_2 + \frac{1}{2} z^2_2) \\
\end{array}$$

We split the $t^2$ coefficients into the two types discussed above: the second fundamental form terms, in the following table,
and the tangent vectors appearing at second order, which are in the table below.
$$\begin{array}{cc}
p_1 & \frac{-1}{2} x^2_2 \otimes y^1_1 \otimes z^1_3 + x^2_1 \otimes y^1_1 \otimes z^1_2 + \frac{-1}{2} x^2_2 \otimes (-y^2_1 + y^2_2) \otimes z^1_2\\
p_3 & (x^3_2 - \frac{1}{2} x^2_2) \otimes y^1_2 \otimes z^2_3 \\
p_4 & (-x^3_1 - \frac{1}{2}x^2_2) \otimes -y^1_1 \otimes z^1_3  \\
p_6 & \frac{1}{2} x^2_2 \otimes y^1_2 \otimes z^2_3 + x^2_1 \otimes y^1_2 \otimes z^2_2 + \frac{1}{2} x^2_2 \otimes (-y^2_1 + y^2_2) \otimes z^2_2 \\
& \\
p_5 & x^2_2 \otimes y^2_1 \otimes (\frac{1}{2} z^1_2 + \frac{1}{2} z^2_2) \\
p_7 &  \\
& \\
p_2 & \\
p_8 & x^2_2 \otimes y^2_2 \otimes ( \frac{1}{2} z^1_2 + \frac{1}{2} z^2_2) \\
\end{array}$$

$$\begin{array}{cc}
p_1 & \frac{-1}{2} x^3_2 \otimes (-y^2_1 + y^2_2) \otimes z^1_3\\

p_3 &  x^3_2 \otimes (y^2_1 + y^2_2) \otimes z^2_3 \\
p_4 & \frac{1}{2} x^3_2 \otimes (y^2_1 + y^2_2) \otimes z^1_3 \\
p_6 & \\
& \\
p_5 & -x^3_2 \otimes y^2_1 \otimes z^2_3 + -x^1_2 \otimes y^2_1 \otimes -z^1_1 \\
p_7 & \\
& \\ 
p_2 & \\
p_8 & x^1_2 \otimes y^2_2 \otimes z^2_1.\\
\end{array}$$
  
Then $T_{BCLRS,3}$ is the sum of the terms in the two tables above.

\section{$T_{BCLRS,4}$}
This algorithm is more complicated and qualitatively different than the others, so we only discuss it briefly.
Note that here the order of the sizes of the matrices are changed: $2\times 4$, $4\times 2$ and $2\times 2$.
\begin{align*}
&p_1(t) =   (x^1_2 - \frac{7}{25} t^3 x^1_3 + x^2_2 - \frac{1}{50} t^3 x^2_3 + t^2 x^2_4) \otimes (\frac{-47}{112} t^4 y^2_2 + \frac{25}{7} t y^3_2 + \frac{8}{21} y^4_1 + t^2 y^4_2) \otimes (-z^1_2 + t^2 \frac{1}{2} z^2_1 + z^2_2)\\
&p_2(t) =   (x^1_2 + \frac{1}{8} x^2_2 - \frac{1}{50} t^3 x^2_3 + \frac{1}{8} t^2 x^2_4) \otimes ( \frac{8}{7} t^4 y^2_2 + \frac{3200}{63} t y^3_2 + \frac{128}{189} y^4_1) \otimes (t^2 \frac{21}{16} z^1_1 + \frac{1}{8} z^1_2 + \frac{13}{16} t^2 z^2_1 - z^2_2)\\
&p_3(t) =   (x^1_2 - \frac{103}{300} x^3 x^1_3 + x^2_2) \otimes (\frac{1}{8} t^5 y^1_2 + \frac{1}{3} t^3 b^2_1 + \frac{25}{16} t^5 y^2_2 + \frac{5}{4} y^3_1 - \frac{1}{3} t y^3_1 - t^3 y^4_2) \otimes (\frac{1}{2} t^2 z^2_1 + z^2_2)\\
&p_4(t) =   (-x^1_2 + x^2_2 + \frac{1}{50} t x^2_3 + x^2_4) \otimes ( -\frac{25}{9} t y^3_2 + \frac{8}{27} y^4_1) \otimes (3t^2 z^1_1 - z^1_2 + \frac{1}{2} t^2 z^2_1 - z^2_2)\\
&p_5(t) =   (x^1_2 + \frac{1}{50} t^3 x^1_3) \otimes (t^3 y^2_1 + \frac{75}{32} y^3_1 - 75 t^2 y^3_2 - t y^4_1) \otimes (t^2 z^1_1 + \frac{2}{3} t^2 z^2_1 - \frac{2}{3} z^2_2)\\
&p_6(t) =  (-\frac{1}{8} x^1_2 + \frac{61}{800} t^3 x^1_3 + x^2_1 - \frac{1}{8} x^2_2) \otimes (5 y^3_1 + t^5 y^1_2) \otimes (\frac{5}{2} t^2 z^1_1 + t^2 z^2_1 + z^2_2)\\
&p_7(t) =   (x^1_2 + \frac{3}{100} t^3 x^1_3 + t^2 x^1_4 - x^2_2 - t^2 x^2_4) \otimes (-\frac{7}{16} t^4 y^2_2 + \frac{1}{3} y^4_1) \otimes (3 t^2 z^1_1 - z^1_2)\\
&p_8(t) =   (x^1_2 + \frac{19}{300} t^3 x^1_3 - x^2_2) \otimes (\frac{1}{3} x^3 y^2_1 + \frac{5}{4} y^3_1 - \frac{1}{3} t y^4_1) \otimes (-\frac{5}{2} t z^2_1  + z^2_2)\\
&p_9(t) =   (x^1_2 - \frac{29}{100} t^3 x^1_3 + t^2 x^1_4 + x^2_2 + t^2 x^2_4) \otimes (\frac{1}{3} y^4_1 + t^2 y^4_2) \otimes (z^1_2)\\
&p_{10}(t) =  (-\frac{5}{16} x^1_2 - 5 x^2_1 + \frac{5}{8} x^2_2) \otimes y^3_1 \otimes (\frac{5}{2} t^2 z^1_1 - 5 t^2 z^2_1 + z^2_2)\\
&p_{11}(t) =   (-x^2_1 + \frac{1}{30} t^3 x^2_3) \otimes (-t^3 y^1_1 + 30 y^3_1) \otimes  z^2_1.
\end{align*}
Then
$$
T_{BCLRS,4}=\frac 1{t^5}
[p_3(t)+p_5(t)+p_6(t)+p_{10}(t)+t(p_1(t)+p_2(t)+p_4(t)+p_{7}(t)+p_{9}(t))+t^2p_{11}(t)].
$$

The limit points are  (ignoring scales which are irrelevant for the geometry):
\begin{align*}
p_1= & (x^1_2 + x^2_2    )        \ot   y^4_1                \ot  (z^1_2 - z^2_2 )        \\
p_2= & (x^1_2 + \frac{1}{8} x^2_2 + \frac{1}{8} x^2_4 )    \ot    y^4_1               \ot  (\frac{1}{8} z^1_2 - z^2_2) \\
p_3= & x^1_2 + x^2_2       \ot    y^3_1               \ot  z^2_2      \\
p_4= & -x^1_2 + x^2_2              \ot   y^4_1                  \ot  (z^1_2 + z^2_2)         \\
p_5= & x^1_2                \ot    y^3_1       \ot    z^2_2 \\
p_6= & (-  x^1_2 + 8x^2_1 -   x^2_2) \ot       y^3_1     \ot   z^2_2 \\
p_7= & (x^1_2 - x^2_2 )       \ot  y^4_1     \ot    z^1_2    \\
p_8= & (x^1_2 - x^2_2)          \ot y^3_1     \ot   z^2_2     \\
p_9= & (x^1_2 + x^2_2 )         \ot  y^4_1     \ot   z^1_2     \\
p_{10}= & (-  x^1_2 - 16x^2_1 + 2 x^2_2) \ot   y^3_1 \ot   z^2_2 \\
p_{11}= &   x^{2}_1     \ot   y^3_1     \ot   z^2_1 . 
\end{align*}
Here $p_3,p_5,p_6,p_{10}$ (the \lq\lq honest\rq\rq\ limit points)  lie on a $\pp 2\times \pp 0\times \pp 0$, 
namely $\BP(\langle x^1_2,x^2_2,x^2_1\rangle\ot y^3_1\ot z^2_2)$, a much
simpler limit configuration than previously. The point $x^2_2$ shows up at zero-th order, whereas in the previous algorithms
only vectors tangent to $x^1_1$ in $Seg(\BP U^*\otimes \BP V)$ showed up at zero-th order. The high order of the algorithm makes its geometry
difficult to analyze.

\section{The Alekseev-Smirnov  border rank algorithm for $M_{\langle 4,2,2\rangle}$}
While this algorithm does not split the matrix multiplication tensor into the sum of two tensors and two algorithms,
it still has features of the other algorithms.

We  rearrange the points and flip the super/subscript of $z$ (the  ordering
in \cite{AlSmir} was 1,2,3,4,5,10,13,6,11,7,9,8,12). We also  modified the derivatives of $p_{12}$ and $p_{13}$,  and the second derivative of $p_{4}$.

\begin{align*}
&p_1(t) = (-t x^1_2 + x^2_2 + x^4_2) \otimes (-y^1_1 + y^2_2) \otimes (- z^1_1+z^1_3 - tz^1_4+t^2 z^2_2)\\
&p_2(t) = (t x^1_2 - x^2_2) \otimes (-t y^2_1 + y^2_2) \otimes (-z^1_1 + z^1_3 - t z^1_4 + tz^2_1)\\
&p_3(t) = (t^2 x^1_1 + tx^1_2 - t x^2_1 - x^2_2 - x^4_2) \otimes  y^1_1  \otimes (z^1_1 - t^2 z^2_2)\\
&p_4(t) = (t^2 x^3_1 -tx^4_1 - tx^1_2 + x^2_2 + x^4_2) \otimes (-y^1_1 + t y^2_1) \otimes (-z^1_3 + t z^1_4)\\
&p_5(t) =  x^4_2  \otimes  y^2_2  \otimes (z^1_1 -  z^1_3 + t z^1_4 - t^2 z^2_2 - t z^2_3 + t^2 z^2_4)\\
& \\
&p_{6}(t) = (t^2 x^1_1 + x^2_2) \otimes (y^1_2 - t^2 y^2_1 + t y^2_2) \otimes  z^1_2 \\
&p_{7}(t) =  x^2_2  \otimes  y^1_2  \otimes (-z^1_1 - t z^2_1 - z^1_2)\\
& \\
&p_8(t) =  x^2_1 \otimes (-ty^1_1 + t y^2_1 + y^1_2) \otimes (-z^1_1 - t z^1_2 + t^2 z^2_2)\\
&p_{9}(t) = (x^2_1 + x^2_2) \otimes (y^1_2 + t y^2_1) \otimes (z^1_1 + t z^2_1)\\
& \\
&p_{10}(t) = (t^2 x^3_1 + t x^3_2 + x^4_2) \otimes (y^1_2 - t y^2_1) \otimes (-z^1_3 + z^2_3)\\
&p_{11}(t) = (t x^3_2 + x^4_2) \otimes (-y^1_2 + t y^2_1 + t y^2_2) \otimes  z^2_3 \\
& \\
&p_{12}(t) = (t^2 x^3_1 + x^4_1) \otimes (ty^1_1 + y^1_2 - t^2 y^2_1) \otimes (z^1_3 + t^2 z^2_4)\\
&p_{13}(t) = (tx^3_2 + x^4_2 - x^4_1) \otimes  y^1_2  \otimes   z^3_1. 
\end{align*}

Then $M_{\langle 2, 2, 4\rangle} = \lim_{t \to 0} \frac{1}{t^{ 2}} \sum p_i(t)$.

The limiting points are (ignoring signs irrelevant for geometry):
\begin{align*}
p_1= & (x^2_2 + x^4_2) \ot  ( y^1_1 - y^2_2)  \ot   (  z^1_1-z^1_3)\\
p_2= &   x^2_2  \ot   y^2_2  \ot   ( z^1_1 - z^1_3) \\
p_3= &   (x^2_2 + x^4_2 ) \ot   y^1_1  \ot   z^1_1 \\
p_4= & (x^2_2 + x^4_2 ) \ot    y^1_1  \ot  z^1_3 \\
p_5= & x^4_2  \ot   y^2_2  \ot  ( z^1_1 -  z^1_3 )\\
& \\
p_{6}= & x^2_2  \ot   y^1_2  \ot   z^2_1 \\
p_{7}= & x^2_2  \ot   y^1_2  \ot  (z^1_1+ z^2_1) \\
p_8= & x^2_1  \ot   y^1_2  \ot  z^1_1 \\
p_{9}= & x^2_1 + x^2_2  \ot   y^1_2  \ot   z^1_1 \\
& \\
p_{10}= & x^4_2  \ot   y^1_2  \ot   (z^1_3 - z^2_3) \\
p_{11}= & x^4_2  \ot  y^1_2  \ot  z^2_3 \\
p_{12}= & x^4_1  \ot   y^1_2  \ot   z^1_3 \\
p_{13}= & (x^4_2 - x^4_1 )  \ot     y^1_2  \ot   z^1_3 .\end{align*}

The terms are grouped as above because there are three independent failures of linear independence:
First $\langle p_1\hd p_5\rangle \cap Seg(\BP A\times \BP B\times \BP C)$ form a BCLR-type configuration of 
special  $(\g,\mu)$ and $(\a,\mu^*)$ lines plus a $\b$-line with rank one $A,C$ elements that intersects 
the special lines, namely
\begin{align*}
L_{(\g,\mu)}= & (x^2_2 + x^4_2) \otimes y^1_1 \otimes (w^1 \otimes \langle u_1, u_3 \rangle)  \\
L_{(\a,\mu^*)}= & (\langle u^2, u^4 \rangle \otimes v_2) \otimes y^2_2 \otimes (z^1_1 -  z^1_3)   \\
L_{\b}= & (x^2_2 + x^4_2)  \otimes \langle y^1_1, y^2_2 \rangle \otimes (z^1_1 -  z^1_3)  .
\end{align*}
In this configuration, the space $B$ plays the role of $A$ in the earlier expressions. 
Then there are two pairs of lines that intersect in a point causing linear dependence (subscript indicates type).
They are $\langle p_6\hd p_9\rangle$
and $\langle p_{10}\hd p_{13}\rangle$,  which are each  contained in a $\pp 2$ spanned by two intersecting lines on the Segre.
\begin{align*}
S_{(\g,\o^*)}= & x^2_2 \otimes y^1_2 \otimes (W^* \otimes u_1)   \\
S_{(\a,\nu)}= & (u^2 \otimes V) \otimes y^1_2 \otimes z^1_1  
\end{align*}
and
\begin{align*}
T_{(\g,\o^*)}= & x^4_2 \otimes y^1_2 \otimes (W^* \otimes u_3)  \\
T_{(\a,\nu)}= & (u^4 \otimes V) \otimes y^1_2 \otimes z^1_3    .
\end{align*}

We use the same notation as above in describing the first and second derivatives.
The first derivatives correspond to:
$$\begin{array}{cccc}
p'_1  & -x^1_2 &  & -z^1_4\\
p'_2  &  x^1_2 & -y^2_1 & -z^1_4 + z^2_1 \\
p'_3  & x^1_2 - x^2_1 & &  \\
p'_4  & -x^4_1 - x^1_2 & y^2_1 & z^1_4 \\
p'_5  &  &  & z^1_4 - z^2_3 \\
p'_{6}  &  & y^2_2 &  \\
p'_{7}  &  &  & -z^1_2 \\
p'_8  &  & -y^1_1 + y^2_1 & -z^1_2 \\
p'_{9}  &  & y^2_1 & z^1_2 \\
p'_{10}  & x^3_2 & -y^2_1 & \\
p'_{11}  & x^3_2 & y^2_1 + y^2_2 &  \\
p'_{12}  &  & y^1_1 &  \\
p'_{13}  & x^3_2 &  & . \\
\end{array}
$$

The second derivatives corresponding to new tangent vectors come from:

$$\begin{array}{cccc}
p''_1  &  &  & z^2_2\\
p''_2  &  &  & \\
p''_3  & x^1_1 & & -z^2_2  \\
p''_4  & x^3_1 &  &  \\
p''_5  &  &  & -z^2_2+z^2_4 \\
p''_{6}  & x^1_1 & -y^2_1 &  \\
p''_{7}  &  &  & \\
p''_8  &  &  & z^2_2 \\
p''_{9}  &  &  &  \\
p''_{10}  & x^3_1 &  & \\
p''_{11}  &  &  &  \\
p''_{12}  & x^3_1 & -y^2_1 & z^2_4  \\
p''_{13}  &  &  & . \\
\end{array}
$$

Note that $p_3, p_4$ lie on $L_{(\g,\mu)}$ and $p_2,p_5$ lie on $L_{(\a,\mu^*)}$ and $p_1$ is chosen exactly so it cancels $p_2 + \ldots + p_5$,
just as with $T_{BCLR}$. 
Note further that  $p_{6}, p_{7} \in S_{(\g,\o^*)}$,  $p_8, p_{9} \in S_{(\a,\nu)}$, 
 $p_{10},p_{11} \in T_{(\g,\o^*)}$, and $p_{12}, p_{13} \in T_{(\a,\nu)}$. 

\begin{remark}    The permutation
  $U\ra U$ (and its induced action $U^*\ra U^*$) exchanging $u_1\leftrightarrow u_3$ and
  $u_2\leftrightarrow u_4$ 
  preserves $M_{\langle 4,2,2\rangle}$. In the algorithm it  switches the role of the  $S$'s and $T$'s and fixes the $L$'s.
\end{remark}

\medskip

The first order derivatives are as follows:

$p_1, \ldots, p_5$ from the $\BP^1 \times \BP^1 \times \BP^1$ contribute the following terms: 
$$\begin{array}{cc}
p_1 & -x^1_2 \otimes (-y^1_1 + y^2_2) \otimes (-z^1_1 + z^1_3) - (x^2_2 + x^4_2) \otimes (-y^1_1 + y^2_2) \otimes  z^1_4 \\
p_2 & x^1_2 \otimes y^2_2 \otimes (-z^1_1 + z^1_3) +  x^2_2 \otimes  y^2_1  \otimes (-z^1_1 + z^1_3)  -x^2_2  \otimes  y^2_2  \otimes (-z^1_4 + z^2_1)\\
p_3 & (x^1_2 - x^2_1) \otimes y^1_1 \otimes z^1_1 \\
p_4 & -( x^4_1 + x^1_2) \otimes  y^1_1 \otimes  z^1_3 - (x^2_2 + x^4_2) \otimes y^2_1 \otimes  z^1_3 - (x^2_2 + x^4_2) \otimes  y^1_1  \otimes z^1_4 \\
p_5 & x^4_2  \otimes y^2_2 \otimes (z^1_4 - z^2_3) \\
\end{array}$$

These sum to $-x^2_2\otimes y^2_1\otimes  z^1_1 - x^2_2 \otimes y^2_2 \otimes z^2_1 - x^2_1 \otimes  y^1_1 \otimes z^1_1 - x^4_1 \otimes y^1_1\otimes z^1_3
- x^4_2 \otimes y^2_1 \otimes  z^1_3 - x^4_2 \otimes y^2_2 \otimes  z^2_3$.

The points from $S_{(\g,\o^*)}$ and $T_{(\g,\o^*)}$ contribute: 
$$\begin{array}{cc}
p_6 & x^2_2 \otimes y^2_2 \otimes z^2_1 \\
p_7 & -x^2_2 \otimes y^1_2 \otimes  z^1_2  \\
p_8 & -x^2_1 \otimes (-y^1_1 + y^2_1) \otimes  z^1_1  -  x^2_1  \otimes y^1_2 \otimes  z^1_2  \\
p_9 & (x^2_1 + x^2_2) \otimes y^2_1 \otimes z^1_1 + (x^2_1 + x^2_2) \otimes y^1_2 \otimes z^1_2. \\
\end{array}$$
These sum to $x^2_2\otimes-y^2_1\otimes -z^1_1 + x^2_2 \otimes y^2_2 \otimes z^2_1 + x^2_1 \otimes  y^1_1 \otimes z^1_1$.

$T_{(\g,\o^*)}$ and $T_{(\a,\nu)}$ contribute: 
$$\begin{array}{cc}
p_{10} & x^3_2 \otimes y^1_2 \otimes (-z^1_3 + z^2_3) - x^4_2 \otimes  y^2_1  \otimes (-z^1_3 + z^2_3) \\
p_{11} & -x^3_2 \otimes  y^1_2 \otimes z^2_3 + x^4_2 \otimes (y^2_1 + y^2_2) \otimes z^2_3 \\
p_{12} & x^4_1 \otimes y^1_1 \otimes z^1_3 \\
p_{13} & x^3_2 \otimes y^1_2 \otimes z^1_3.\\
\end{array}$$
These sum to $x^4_1 \otimes y^1_1 \otimes z^1_3 + x^4_2 \otimes y^2_1 \otimes z^1_3 + x^4_2 \otimes y^2_2 \otimes z^2_3$.

\medskip

The second order terms are as follows:

Terms from $L_1, L_2$ in the second fundamental form are: 
$$\begin{array}{cc}
p_1 &  x^1_2  \otimes (-y^1_1 + y^2_2) \otimes  z^1_4  \\
p_2 &  -x^1_2  \otimes  y^2_1  \otimes (-z^1_1 + z^1_3) +  x^1_2  \otimes y^2_2 \otimes (-z^1_4 + z^2_1)  +x^2_2  \otimes  y^2_1  \otimes (-z^1_4 + z^2_1) \\
p_3 & \\
p_4 &  ( x^4_1 + x^1_2) \otimes y^2_1 \otimes  z^1_3  + ( x^4_1 + x^1_2) \otimes  y^1_1  \otimes z^1_4  + (x^2_2 + x^4_2) \otimes y^2_1 \otimes z^1_4 \\
p_5 & .
\end{array}$$

These sum to $x^1_2 \otimes y^2_2 \otimes z^2_1 +  x^1_2 \otimes y^2_1 \otimes z^1_1 + x^2_2 \otimes y^2_1 \otimes z^2_1 + x^4_1 
\otimes y^2_1 \otimes z^1_3 + x^4_1 \otimes y^1_1 \otimes z^1_4 + x^4_2 \otimes y^2_1 \otimes z^1_4$.

The $t^2$ terms from tangent vectors to  $L_1$ and $L_2$ are:  
$$\begin{array}{cc}
p_1 &(x^2_2 + x^4_2) \otimes (-y^1_1 + y^2_2) \otimes z^2_2\\
p_2 &\\
p_3 & x^1_1 \otimes y^1_1 \otimes z^1_1 + ( x^2_2+ x^4_2) \otimes y^1_1 \otimes  z^2_2  \\
p_4 & x^3_1 \otimes y^1_1 \otimes z^1_3 \\
p_5 & x^4_2 \otimes y^2_2 \otimes (-z^2_2 + z^2_4) \\
\end{array}$$

These sum to $x^1_1 \otimes y^1_1 \otimes z^1_1 + x^2_2 \otimes y^2_2 \otimes z^2_2 + x^4_2 \otimes y^2_2 \otimes z^2_4 + x^3_1 \otimes y^1_1 \otimes z^1_3$.

Terms from $S_{(\g,\o^*)}, S_{(\a,\nu)}$ in the second fundamental form are: 
$$\begin{array}{cc}
p_6 & \\
p_7 & \\
p_8 & -x^2_1 \otimes (-y^1_1 + y^2_1) \otimes  z^1_2  \\
p_9 & (x^2_1 + x^2_2) \otimes  y^2_1  \otimes z^1_2 \\
\end{array}$$

These sum to $x^2_1 \otimes y^1_1 \otimes z^1_2 + x^2_2 \otimes y^2_1 \otimes z^1_2$.

The $S_{(\g,\o^*)}, S_{(\a,\nu)}$   $t^2$ terms from tangent spaces are: 
$$\begin{array}{cc}
p_6 & x^1_1 \otimes y^1_2 \otimes z^2_1 - x^2_2 \otimes  y^2_1  \otimes z^2_1 \\
p_7 & \\
p_8 & x^2_1 \otimes y^1_2 \otimes z^2_2 \\
p_9 & \\
\end{array}$$

These sum to $x^2_1 \otimes y^1_2 \otimes z^2_2 - x^2_2 \otimes y^2_1 \otimes z^2_1 + x^1_1 \otimes y^1_2 \otimes z^2_1$.

Terms from $T_{(\g,\o^*)}$ and $T_{(\a,\nu)}$ in the second fundamental form are: 
$$\begin{array}{cc}
p_{10} & -x^3_2 \otimes  y^2_1  \otimes (-z^1_3 + z^2_3) \\
p_{11} & x^3_2 \otimes (y^2_1 + y^2_2) \otimes z^2_3 \\
p_{12} &  \\
p_{13} & \\
\end{array}$$
These sum to $x^3_2 \otimes y^2_1 \otimes z^1_3 + x^3_2 \otimes y^2_2 \otimes z^2_3$.

The $T_{(\g,\o^*)}$ and $T_{(\a,\nu)}$   $t^2$ terms from tangent spaces are:  
$$\begin{array}{cc}
p_{10} & x^3_1 \otimes y^1_2 \otimes (-z^1_3 + z^2_3) \\
p_{11} & \\
p_{12} & x^3_1 \otimes y^1_2 \otimes z^1_3 - x^4_1 \otimes  y^2_1  \otimes z^1_3 + x^4_1 \otimes y^1_2 \otimes z^2_4 \\
p_{13} & \\
\end{array}$$
These sum to $x^3_1 \otimes y^1_2 \otimes z^2_3 - x^4_1 \otimes y^2_1 \otimes z^1_3 + x^4_1 \otimes y^1_2 \otimes z^2_4$.

\section{Brief remarks on Smirnov's border rank $20$ algorithm for
$M_{\langle 3,3,3\rangle}$}\label{m333}

The \cite[Table 6]{MR3146566} border rank algorithm for $M_{\langle 3,3,3\rangle}$ expressed as a tensor is: 
\begin{tiny}
\begin{align*}
&p_1(t) = t^{-6} (t^3 x^1_1 - t^6 x^1_3 + t x^3_1 +  x^3_3) \otimes (t^2 y^1_1  + t^3 y^1_2 + t y^2_1 +  y^3_1 + 2 t^5 y^3_3) \otimes (z^2_1 - t^{4} z^1_2 + t^6 z^1_3)\\
&p_2(t) = t^{-5} (t^5 x^1_3 +  x^2_1 - t x^2_2 - t^3 x^3_3) \otimes (t^4 y^2_3 + y^3_3) \otimes (-t^3 z^1_1 +  z^3_1 - t^3 z^1_2 + t^2 z^2_2 + t^3 z^3_2 - t^2 z^3_3)\\
&p_3(t) = t^{-5} (-t x^2_1 + t^2 x^2_2 - x^3_2) \otimes (y^2_1 + t y^2_2 - t^{3} y^2_3) \otimes (t^3 z^1_1 - z^3_1 + (t^2-t^6) z^1_2 + t^2 z^3_3)\\
&p_4(t) = t^{-5} (-t^2 x^1_1 + t^3 x^1_2 - x^3_1 + t x^3_2) \otimes (t^2 y^1_1 + t y^2_1 + y^3_1) \otimes (-t z^1_1 +  z^2_1 + t^6 z^1_3)\\
&p_5(t) = t^{-6} (-t^5 x^1_3 + x^2_1 + x^3_3) \otimes (t^2 y^1_1 + t^3 y^1_2 + t y^2_1 + y^3_1 + t y^3_2 + t^5 y^3_3) \otimes (- z^2_1 + t^4 z^1_2)\\
&p_6(t) = t^{-6} (x^3_1 - t x^3_2) \otimes (t^4 y^1_2 + y^2_1) \otimes (t^2 z^1_1 - t^2 z^2_1 - z^3_1 - (t^4 + t^5)z^1_3 + t^4 z^2_3 + t^2 z^3_3)\\
&p_7(t) = t^{-4} (t^3 x^1_1 - 2t^4 x^1_2 + x^3_1 - t x^3_2) \otimes (t^2 y^1_1 + t^4 y^1_3 + y^2_1) \otimes (- z^1_1)\\
&p_8(t) = t^{-6} (t^3 x^1_1 - t^5 x^1_3 -  x^2_1 + t^2 x^2_3 + t x^3_1) \otimes (t^3 y^1_3 + y^3_1) \otimes ( z^3_1 + (t^4 - t^3)z^1_2)\\
&p_9(t) = t^{-5} (- x^2_1 + t x^2_2) \otimes (t^2 y^2_2 + y^3_3) \otimes (-t^3 z^1_1 +  z^3_1 - t^3 z^1_2 + (t^2-t^5)z^2_2 - t^2 z^3_3)\\
&p_{10}(t) = t^{-5} (x^3_2) \otimes (y^2_1 + t^4 y^1_3 + t y^2_2) \otimes (t^3 z^1_1 - t^2 z^2_1 - z^3_1 + t^3 z^1_2 + t^4 z^2_3 + t^2 z^3_3)\\
&p_{11}(t) =  t^{-4} (x^2_1 - t^2 x^2_3) \otimes (t y^1_3 + y^2_3 + y^3_2) \otimes (-t^3 z^1_2 +  z^3_1 + t^3 z^3_2)\\
&p_{12}(t) = t^{-6} (- x^2_1 + t^4 x^1_2 + t^2 x^2_3 + t^2 x^3_2) \otimes (-t^2 y^2_3 + y^3_1) \otimes (t^2 z^1_2 - z^3_1)\\
&p_{13}(t) = t^{-6} (-t^2 x^2_1 + t^3 x^2_2 +  x^3_1 - t x^3_2) \otimes y^2_1 \otimes (-t^3 z^1_1 + z^3_1 + t^6 z^1_2 - t^2 z^3_3)\\
&p_{14}(t) = t^{-5} (-t^2 x^1_1 + t^3 x^1_2 + t^4 x^1_3 - x^3_1 + t x^3_2) \otimes (y^3_1) \otimes (t z^1_1 + z^3_1 - t^3 z^1_2)\\
&p_{15}(t) = t^{-6} (x^3_3) \otimes (y^3_2 - t^{4} y^3_3) \otimes (t z^2_1 -  z^3_1 - t^5 z^1_2 - t^2 z^2_2 - t^3 z^3_2 + t^6 z^2_3)\\
&p_{16}(t) = t^{-6} (-t^4 x^1_2 + x^2_1 - t^2 x^3_2) \otimes (t^2 y^1_1 + t y^2_1 - t^2 y^2_2 + y^3_2) \otimes (z^2_1)\\
&p_{17}(t) = t^{-4} (x^3_1) \otimes (-t^2 y^1_2 + t^4 y^1_3 + y^2_1 - 2 t^4 y^3_3) \otimes ( z^2_1 + t^3 z^1_3 - t^2 z^2_3)\\
&p_{18}(t) = t^{-2} (- x^2_1 + t x^2_2 + t^2 x^2_3) \otimes (t y^2_3 + y^3_3) \otimes (t^2 z^2_2 + z^3_2)\\
&p_{19}(t) = t^{-6} (-t^2 x^2_1 + t^4 x^2_3 + x^3_3) \otimes (y^3_2) \otimes (z^3_1 + t^2 z^2_2 + t^3 z^3_2)\\
&p_{20}(t) = t^{-5} (x^2_1) \otimes (t^2 y^1_2 + t y^2_2 + y^3_2 + t^4 y^3_3) \otimes ( z^2_1 + t^3 z^2_2).
\end{align*}
\end{tiny}

The limit points are 
$$\begin{array}{c|c|c|c}
 p_1 & x^3_3 & y^3_1 & z^2_1  \\
 p_2=p_9 & x^2_1 & y^3_3 & z^3_1 \\
 p_3=p_{10} & -x^3_2 & y^2_1 & -z^3_1 \\
 p_4 & -x^3_1 & y^3_1 & z^2_1 \\
 p_5 & x^2_1 + x^3_3 & y^3_1 & -z^2_1 \\
 p_6 & x^3_1 & y^2_1 & -z^3_1 \\
 p_7 & x^3_1 & y^2_1 & -z^1_1 \\
 p_8& -x^2_1 & y^3_1 & z^3_1 \\
 p_{11} & x^2_1 & y^2_3 + y^3_2 & z^3_1 \\
 p_{12} & -x^2_1 & y^3_2 & -z^3_1 \\
 p_{13} & x^3_1 & y^2_1 & z^3_1 \\
 p_{14} & -x^3_1 & y^3_1 & z^3_1 \\
 p_{15}=p_{19} & x^3_3 & y^3_2 & -z^3_1 \\
 p_{16}=p_{20} & x^2_1 & y^3_2 & z^2_1 \\
 p_{17} & x^3_1 & y^2_1 & z^2_1 \\
 p_{18} & -x^2_1 & y^3_3 & z^3_2 \\
\end{array}$$

In addition to the duplication of points, the limit points are in a very degenerate configuration.
For example all the $z$ points lie on a $\pp 2$ of type $\o^*$, all the $x$ points lie on a $\pp 3$ and all the $y$
points on a $\pp 4$.

\section{Remarks on the uniqueness of the 
BCLR  border rank algorithms}\label{uniquerems}
Let $T=T_1+\cdots + T_r$ be a 
  rank $r$ expression for a tensor  $T\in A\ot B\ot C$.    The tensor $T$ is said to be
{\it identifiable} if the $[T_j]$ are unique, i.e., $T$ is
not in the span of any other collection of $r$ rank one tensors (up to scale). 
For border rank algorithms it will be more useful to define a weaker notion
of identifiability: we will say $T$ is {\it Grassmann identifiable} if
$\langle T_1\hd T_r\rangle$ is the unique $r$-plane spanned by rank
one tensors that contains $T$, and {\it Grassmann border identifiable} if there exists 
a unique $E\in G(r,A\ot B\ot C)$ that is the limit of some $\langle T_1(t)\hd T_r(t)\rangle$ with
$T_j(t)$ of rank one. 

Tensors with symmetry are rarely Grassmann border identifiable because their symmetry
group acts on the Grassmannian and will move the algorithm to other algorithms.
In what follows we discuss the action of the symmetry group on 
the limiting $5$-plane $E^{BCLR}$  for the BCLR tensor,  
 and the limiting $10$-plane for the sum of two BCLR-type tensors
glued together to form a border rank algorithm for $M_{\langle 3,2,2\rangle}$, which we denote by 
$\tilde E^{BCLR}$. 
 
 We expect the following information to be useful in constructing new
 algorithms for $M_{\langle n,2,2\rangle}$.
 
The role $x^1_1$  in a BCLR-type algorithm can be played by any element of
$Seg(\pp 1\times\pp{n-1})$, so the glued together algorithms come
in families parametrized by this Segre variety. 

The choice of an element to blank out determines a split $U=\BC^1\op \BC^{n-1}$ and $V=\BC^1\ot \BC^1$. The subgroup
preserving such a splitting is $G_U\times T_V\times SL(W)$, where $T_V\subset SL(V)$ denotes the diagonal matrices and 
$$
G_U=\begin{pmatrix} * &  & \\
& * & * \\
& * & *
\end{pmatrix} \cap SL(U)
$$
where the blocking is $(1,n-1)\times (1,n-1)$. 

A border rank algorithm for $M_{\langle n,2,2\rangle}$ obtained from two BCLRS-type algorithms has a further
splitting $\BC^{n-1}=\BC^{m-1}\op \BC^{n-m}$.

The following was shown via a computer calculation of the Lie algebra of the stabilizers by F. Gesmundo:

\begin{proposition}
$\tilde E^{BCLR}\in G(10,A\ot B\ot C)$ has a $10$ dimensional orbit under $SL(U)\times SL(V) \times SL(W)$.
The connected component of the identity of its stabilizer is $T_U\times T_V\times T_W$.

$  E^{BCLR}\in G(5,A\ot B\ot C)$ has a $2$-dimensional orbit under $G_U\times T_V\times SL(W)$.
The connected component of the identity of its stabilizer is $T_U\times T_V\times T_W$.
\end{proposition}

\bibliographystyle{amsplain}
 
\bibliography{Lmatrix}

\def\cdprime{$''$} \def\cprime{$'$} \def\cprime{$'$} \def\cprime{$'$}
  \def\Dbar{\leavevmode\lower.6ex\hbox to 0pt{\hskip-.23ex \accent"16\hss}D}
  \def\cprime{$'$} \def\cprime{$'$} \def\cdprime{$''$} \def\cprime{$'$}
  \def\cprime{$'$} \def\Dbar{\leavevmode\lower.6ex\hbox to 0pt{\hskip-.23ex
  \accent"16\hss}D} \def\cprime{$'$} \def\cprime{$'$} \def\cprime{$'$}
  \def\cprime{$'$} \def\Dbar{\leavevmode\lower.6ex\hbox to 0pt{\hskip-.23ex
  \accent"16\hss}D} \def\cprime{$'$} \def\cprime{$'$}
\providecommand{\bysame}{\leavevmode\hbox to3em{\hrulefill}\thinspace}
\providecommand{\MR}{\relax\ifhmode\unskip\space\fi MR }
\providecommand{\MRhref}[2]{%
  \href{http://www.ams.org/mathscinet-getitem?mr=#1}{#2}
}
\providecommand{\href}[2]{#2}
\begin{thebibliography}{10}

\bibitem{AlSmir}
V.~B. Alekseev and A.~V. Smirnov, \emph{On the exact and approximate bilinear
  complexities of multiplication of 4×2 and 2×2 matrices}, Proceedings of the
  Steklov Institute of Mathematics \textbf{282} (2013), no.~1, 123--139.

\bibitem{MR534068}
Dario Bini, Milvio Capovani, Francesco Romani, and Grazia Lotti,
  \emph{{$O(n\sp{2.7799})$} complexity for {$n\times n$} approximate matrix
  multiplication}, Inform. Process. Lett. \textbf{8} (1979), no.~5, 234--235.
  \MR{MR534068 (80h:68024)}

\bibitem{MR3239293}
Jaroslaw Buczy{\'n}ski and J.~M. Landsberg, \emph{On the third secant variety},
  J. Algebraic Combin. \textbf{40} (2014), no.~2, 475--502. \MR{3239293}

\bibitem{HILsec}
Jon Hauenstein, Christian Ikenmeyer, and J.M. Landsberg, \emph{Equations for
  lower bounds on the border rank}, Exper. Math. (to appear), arXiv
  arXiv:1305.0779.

\bibitem{MR2188132}
J.~M. Landsberg, \emph{The border rank of the multiplication of {$2\times2$}
  matrices is seven}, J. Amer. Math. Soc. \textbf{19} (2006), no.~2, 447--459.
  \MR{2188132 (2006j:68034)}

\bibitem{MR2865915}
\bysame, \emph{Tensors: geometry and applications}, Graduate Studies in
  Mathematics, vol. 128, American Mathematical Society, Providence, RI, 2012.
  \MR{2865915}

\bibitem{LMborderalg}
J.M. Landsberg and Mateusz Michalek, \emph{On the geometry of border rank
  algorithms}, preprint.

\bibitem{v011a011}
Joseph~M. Landsberg and Giorgio Ottaviani, \emph{New lower bounds for the
  border rank of matrix multiplication}, Theory of Computing \textbf{11}
  (2015), no.~11, 285--298.

\bibitem{MR623057}
A.~Sch{\"o}nhage, \emph{Partial and total matrix multiplication}, SIAM J.
  Comput. \textbf{10} (1981), no.~3, 434--455. \MR{MR623057 (82h:68070)}

\bibitem{MR3146566}
A.~V. Smirnov, \emph{The bilinear complexity and practical algorithms for
  matrix multiplication}, Comput. Math. Math. Phys. \textbf{53} (2013), no.~12,
  1781--1795. \MR{3146566}

\bibitem{MR3343116}
\bysame, \emph{A bilinear algorithm of length 22 for approximate multiplication
  of {$2\times 7$} and {$7\times 2$} matrices}, Comput. Math. Math. Phys.
  \textbf{55} (2015), no.~4, 541--545. \MR{3343116}

\bibitem{Strassen505}
V.~Strassen, \emph{Rank and optimal computation of generic tensors}, Linear
  Algebra Appl. \textbf{52/53} (1983), 645--685. \MR{85b:15039}

\bibitem{Strassen493}
Volker Strassen, \emph{Gaussian elimination is not optimal}, Numer. Math.
  \textbf{13} (1969), 354--356. \MR{40 \#2223}

\bibitem{Terracini1}
A.~Terracini, \emph{Sulla $v_k$ per cui la varieta degli $s_h(h+1)$-seganti ha
  dimensione minore dell'ordinario}, Rend. Circ. Mat. Palermo \textbf{31}
  (1911), 392--396.

\bibitem{win553}
S.~Winograd, \emph{On multiplication of {$2\times 2$} matrices}, Linear Algebra
  and Appl. \textbf{4} (1971), 381--388. \MR{45 \#6173}

\end{thebibliography}

\end{document}